\documentclass[11pt]{article}

\usepackage{dsfont,fullpage}
\usepackage{amsthm,amsmath,amssymb,enumerate,dsfont,amsbsy}

\newcommand{\tst}{\textstyle}

\newcommand{\m}[1]{\mathcal{#1}}
\newcommand{\bb}[1]{\mathds{#1}}
\newcommand{\ur}[1]{\mathrm{#1}}

\newcommand{\eps}{\varepsilon}

\newcommand{\oh}{\frac{1}{2}}

\newcommand{\R}{\bb R}
\newcommand{\N}{\bb N}

\renewcommand{\d}{\ur{d}}
\newcommand{\p}{\partial}

\renewcommand{\theta}{{\vartheta}}
\renewcommand{\O}{{\Omega}}
\newcommand{\oO}{\overline{\O}}
\newcommand{\pO}{\p \O}

\newcommand{\sia}[1]{{\mathsf{#1}}_i^\alpha}
\newcommand{\siha}[1]{{\mathsf{#1}}_i^\ha}

\newcommand{\hil}[1]{\hat{#1}_i^\ell}

\newcommand{\il}[1]{{#1}_i^\ell}
\newcommand{\ik}[1]{{#1}_i^k}
\newcommand{\kz}[1]{{#1}^k}

\newcommand{\iko}[1]{{#1}_i^{k+1}}
\newcommand{\ko}[1]{{#1}^{k+1}}
\newcommand{\ia}[1]{{#1}_i^\alpha}

\newcommand{\bia}[1]{\bar{#1}_i^\alpha}
\newcommand{\bbia}[1]{\bar{\bar{#1}}_i^\alpha}
\newcommand{\biha}[1]{\bar{#1}_i^\ha}
\newcommand{\bbiha}[1]{\bar{\bar{#1}}_i^\ha}

\newcommand{\Id}{{\sf Id}}

\renewcommand{\P}{{P_i}}

\newcommand{\bllangle}{{\bigl\langle \hspace{-0.9mm} \bigr\langle}}
\newcommand{\brrangle}{{\bigl\rangle \hspace{-0.9mm} \bigr\rangle}}
\newcommand{\llangle}{{\langle \! \langle}}
\newcommand{\rrangle}{{\rangle \! \rangle}}
\newcommand{\ha}{{\hat{\alpha}}}

\newtheoremstyle{assumption}{6pt}{6pt}{\rm}{}{\sffamily}{ }{ }{}
\theoremstyle{assumption}
\newtheorem{assumption}{\sc Assumption}

\newtheoremstyle{definition}{6pt}{6pt}{\rm}{}{\sffamily}{ }{ }{}
\theoremstyle{definition}
\newtheorem{definition}{\sc Definition}

\newtheoremstyle{lemma}{6pt}{6pt}{\rm}{}{\sffamily}{ }{ }{}
\theoremstyle{lemma}
\newtheorem{lemma}{\sc Lemma}

\newtheoremstyle{example}{6pt}{6pt}{\rm}{}{\sffamily}{ }{ }{}
\theoremstyle{example}
\newtheorem{example}{\sc Example}

\newtheoremstyle{theorem}{6pt}{6pt}{\rm}{}{\sffamily}{ }{ }{}
\theoremstyle{theorem}
\newtheorem{theorem}{\sc Theorem}

\newtheoremstyle{remark}{6pt}{6pt}{\rm}{}{\sffamily}{ }{ }{}
\theoremstyle{remark}
\newtheorem{remark}{\sc Remark}

\renewcommand{\a}{\alpha}

\newcommand{\abs}[1]{\left|#1\right|}

\newcommand{\tends}{\rightarrow}

\newcommand{\norm}[1]{\left\|#1\right\|}

\newcommand{\half}{\frac{1}{2}}

\newcommand{\pair}[2]{\langle #1,#2 \rangle}

\newcommand{\calT}{\mathcal{T}}

\newcommand{\tia}[1]{\tilde{#1}^{\a}}
\newcommand{\ttia}[1]{\tilde{\tilde{#1}}^{\a}}
\newcommand{\tha}[1]{\tilde{#1}^{\ha}}
\newcommand{\ttha}[1]{\tilde{\tilde{#1}}^{\ha}}

\title{ $\boldsymbol{L^2(H^1_\gamma)}$ \bf Finite Element Convergence for Degenerate Isotropic Hamilton--Jacobi--Bellman Equations\\[3mm]}
\author{Max Jensen}

\begin{document}

\maketitle

\bigskip

\begin{abstract}
{In this paper we study the convergence of monotone $P1$ finite element methods for fully nonlinear Hamilton--Jacobi--Bellman equations with degenerate, isotropic diffusions. The main result is strong convergence of the numerical solutions in a weighted Sobolev space $L^2(H^1_\gamma(\O))$ to the viscosity solution without assuming uniform parabolicity of the HJB operator.}
{finite element methods, degenerate partial differential equations, Hamilton--Jacobi--Bellman equations, viscosity solutions.}
\end{abstract}


\bigskip

\section{Introduction} \label{sec:introduction}

Hamilton-Jacobi-Bellman (HJB) equations characterise the value functions of optimal control problems. For a wide range of control problems one can compute optimal control policies from the {\em partial derivatives} of the value function.

An important tool in the analysis of HJB equations and their numerical approximations is the concept of viscosity solutions. Its definition is based on sign information on function values of candidate solutions, leading typically to proofs of $L^\infty$ convergence of numerical methods, cf. \cite{BS91}. It is more difficult to prove convergence in other norms if solely viscosity solutions are used.

The use of weak solution, familiar from semilinear differential equations, in the context of Hamilton-Jacobi-Bellman equations is delicate because often uniqueness cannot be ensured. However, we believe that combining the notions of viscosity and weak solution is attractive for numerical analysis: the former to deal with uniqueness and the later to study convergence of partial derivatives.

In \cite{JS13} the uniform convergence of $P1$ finite element approximations to the viscosity solutions of isotropic, degenerate parabolic HJB equations was shown. In addition $L^2(H^1)$ convergence was demonstrated, under the assumption that the HJB equation is uniformly parabolic. In this paper we remove the assumption of uniform parabolicity and verify that strong convergence in weighted $L^2(H^1_\gamma)$ spaces can be maintained, see Theorem \ref{thm:main} below. Also the condition that the $d$-dimensional Lebesgue measure of the boundary of the zero level set of the value function has to vanish is not needed anymore.

Our approach uses coercivity properties of the HJB operator. An alternative technique to control derivative terms is proposed in \cite{SS}, where HJB equations satisfying Cordes conditions are discretised. A general review of the recent advances in the discretesation of fully nonlinear equations is given in \cite{FGN13}.

In Section \ref{sec:problemstatement} we introduce the Bellman equation, while in Section \ref{sec:scheme} the numerical method is defined. Section \ref{sec:review} is concerned with uniform convergence to the value function. In Section \ref{sec:projection} a nonlinear projection operator is constructed and analysed, which preserves positivity and boundary conditions of the viscosity solution. Section \ref{sec:coer} is concerned with the coercivity of the continuous and discrete Bellman operators. The main result of strong convergence in a weighted Sobolev space is proved in Section \ref{sec:gradient}. Finally, in Section \ref{sec:art_diff} assumptions of the prior analysis are translated into concrete parameter values for the method of artificial viscosity.

\section{Problem statement} \label{sec:problemstatement}

Let $\O$ be a bounded Lipschitz domain in $\R^d$, $d \ge 2$. Let $A$ be a compact metric space and let
\[
A \to C(\oO) \times C(\oO, \R^d) \times C(\oO) \times C(\oO),\; \alpha \mapsto (a^\alpha,  b^\alpha, c^\alpha, f^\alpha)
\]
be continuous, such that the families of functions $\{a^{\a}\}_{\a\in A}$, $\{b^{\a} \}_{\a\in A}$, $\{c^{\a}\}_{\a\in A}$ and $\{f^\alpha\}_{\a\in A}$ are equi-continuous. Consider the bounded linear operators 
\[
L^\alpha : \; H^2(\O) \to L^2(\O), \; w \mapsto - a^\alpha \, \Delta w + b^\alpha \cdot \nabla w + c^\alpha \, w, \qquad \a\in A.
\]
We assume that $a^\alpha \ge 0$, i.e.~that all $L^\alpha$ are of degenerate elliptic. Furthermore, suppose that pointwise $f^\alpha \ge 0$. Then
\begin{align} \label{Lbounds}
\sup_{\alpha \in A} \| \, (a^\alpha,  b^\alpha, c^\alpha, f^\alpha) \,
\|_{C(\oO) \times C(\oO, \R^d) \times C(\oO) \times C(\oO)} < \infty,
\end{align}
and also $\sup_{\alpha \in A} \| L^\alpha \|_{H^2(\O) \to L^2(\O)} < \infty.$ Let the final-time data $ v_T \in C(\oO) $ be non-negative, that is $ v_T \geq 0 $ on $\oO$, and let $v_T$ satisfy homogeneous boundary conditions on $\pO$. For smooth $w$, let 
\[
H w := \sup_\alpha (L^\alpha w - f^\alpha),
\]
where the supremum is applied pointwise. The HJB equation considered is
\begin{subequations}\label{eq:HJB_ibvp}
\begin{alignat}{2}
-\p_t v + H v &=0 	&	&\qquad\text{in } \Omega_T := (0,T)\times\O,\label{eq:HJB_ibvp_pde}\\
v&=0 		  	&	&\qquad\text{on }(0,T)\times\pO,\label{eq:HJB_ibvp_lateral}\\
v&=v_T		  	&	&\qquad\text{on }\{T\}\times\oO.\label{eq:HJB_ibvp_final}
\end{alignat}
\end{subequations}

\begin{definition}[\cite{BP87,FS06}]\label{def:viscosity_sol}
An upper semi-continuous (lower semi-continuous) function $v : [0,T] \times \oO \to \R $ is a viscosity subsolution (supersolution) of
\begin{align}\label{HJB}
- \p_t v + H v = 0 \quad\text{on }\O_T,
\end{align}
if for any $w\in C^{\infty}(\R \times\R^d)$ such that $v-w$ has a strict local maximum (minimum) at $(t,x)\in(0,T)\times\O$ with $v(t,x)=w(t,x)$, gives $-\p_t w(t,x)+Hw(t,x) \leq 0$, (greater than or equal to $0$). If $v\in C([0,T]\times\oO)$ is both a viscosity subsolution and a supersolution of equation \eqref{HJB}, then $v$ is called a viscosity solution.
\end{definition}

The viscosity solution of \eqref{eq:HJB_ibvp} is understood to be a viscosity solution of the PDE \eqref{eq:HJB_ibvp_pde}, in the sense of Definition~\ref{def:viscosity_sol}, that satisfies pointwise the boundary conditions \eqref{eq:HJB_ibvp_lateral} and \eqref{eq:HJB_ibvp_final}. Owing to the definition of viscosity solutions, $v$ is a continuous function.

\section{The numerical scheme} \label{sec:scheme}

Let $V_i$, $i\in\N$, be a sequence of piecewise linear, conforming, shape-regular finite element spaces with nodes $\il y$. Here $\ell$ is the index ranging over the nodes of the finite element mesh $\m T_i$. Let $V_i^0\subset V_i$ be the subspace of functions which satisfy homogeneous Dirichlet conditions on $\pO$. It is convenient to assume that $\il y \in \O$ for $\ell \leq N_i := \dim V_i^0$; i.e.~the index $\ell$ first ranges over internal nodes and then over boundary nodes. The associated hat functions are denoted $\il \phi$; that is $\il \phi \in V_i$ and $\il \phi(y_i^l) = 1$ if $l = \ell$, otherwise $\il \phi(y_i^l) = 0$. Set $\hil \phi := \il \phi / \| \il \phi \|_{L^1(\O)}$. Thus, the $\il \phi$ are normalised in the $L^\infty$ norm whilst the $\hil \phi$ are normalised in the $L^1$ norm. The mesh size, i.e. the largest diameter of an element, is denoted $\Delta x_i$. It is assumed that $\Delta x_i \to 0$ as $i \to \infty$.

Let $h_i$ be the (uniform) time step size used in conjunction with $V_i$, with $T / h_i \in \N$, and let $\ik s$ be the $k$th time step at the refinement level~$i$. It is assumed that $h_i \to 0$ as $i \to \infty$. The set of time steps is $S_i := \bigl\{ \ik s : k = 0, \ldots, {\textstyle T / h_i} \bigr\}$. Let the $\ell$th entry of $d_i w( \ik s, \cdot)$ be
\[
(d_i w( \ik s, \cdot))_\ell = \frac{w( \iko s, \il y) - w(\ik s, \il y)}{h_i}.
\]
For each $\alpha$ and $i$, we introduce operators $\ia E$ and $\ia I$ to break $L^\alpha$ into an explicit and implicit part: 
\begin{align*}
\ia E &: \; H^2(\O) \to L^2(\O), \; w \mapsto - \bia a \, \Delta w + \bia b \cdot \nabla w + \bia c \, w,\\
\ia I &: \; H^2(\O) \to L^2(\O), \; w \mapsto - \bbia a \, \Delta w + \bbia b \cdot \nabla w + \bbia c \, w,
\end{align*}
with continuous
\begin{align} \label{conta} \begin{array}{ll}
A \to C(\oO) \times C(\oO, \R^d) \times C(\oO),&\; \alpha \mapsto (\bia a, \bia b, \bia c),\\[1mm]
A \to C(\oO) \times C(\oO, \R^d) \times C(\oO),&\; \alpha \mapsto (\bbia a, \bbia b, \bbia c).
\end{array} \end{align}
It is required that $\bia c$ and $\bbia c$ are non-negative and that there is $C \in \R$ such that
\begin{align}\label{react}
 \| \bia c \|_{L^\infty} + \| \bbia c \|_{L^\infty} \leq C , \qquad \forall\,i\in\N,\;\forall\,\a\in A.
\end{align}
Also, find for each $i$ a non-negative $\ia f$ which approximates $f^\alpha$: $\ia f \approx f^\alpha$. The conceptual statements $L^\alpha \approx \ia E + \ia I$ and $f^\alpha \approx \ia f$ are made precise as follows:

\begin{assumption} \label{consistency}
For all sequences of nodes $(\il y)_{i \in \N}$, where in general $\ell = \ell(i)$ depends on $i$:
\begin{align*}  
\lim_{i \to \infty} \sup_{\a\in A} \bigl( \bigl\| a^\alpha - \bigl( \bia a(\il y) + \bbia a(\il y) \bigr) \bigr\|_{L^\infty({\rm supp} \, \hil \phi)} + \bigl\| b^\alpha - \bigl( \bia b + \bbia b \bigr) \bigr\|_{L^\infty(\O,\R^d)} & \\
 +  \bigl\| c^\alpha - \bigl( \bia c + \bbia c \bigr) \bigr\|_{L^\infty(\O)} + \bigl\| f^\alpha - \ia f \bigr\|_{L^\infty(\O)} & \bigr) = 0.
\end{align*}

\end{assumption}

Let $\pair{\cdot}{\cdot}$ denote the standard inner product for both of the spaces $L^2(\O)$ and $L^2(\O,\R^d)$, the two cases being distinguished by the arguments of the inner product. Consider the following discretisation of $\ia E$ and $\ia I$ by operators $\sia E$ and $\sia I$ that map $H^1(\O)$ to $\R^{N_i}$: for $w \in H^1(\O)$, $\ell \in \{ 1, \dots, N_i = \dim V_i^0 \} $,
\begin{subequations}\label{eq:discreteop}
\begin{align}
(\sia E w)_\ell & := \bia  a(\il y) \langle \nabla w, \nabla \hil \phi \rangle + \langle \bia  b \cdot \nabla w + \bia  c \, w, \hil \phi \rangle,\\
(\sia I w)_{\ell} & := \bbia a(\il y) \langle \nabla w, \nabla \hil \phi \rangle + \langle \bbia b \cdot \nabla w + \bbia c w, \hil \phi \rangle,\\
(\sia F)_\ell & := \langle \ia f, \hil \phi \rangle.
\end{align}
\end{subequations}
Throughout this work, we identify $\sia E $ and $\sia I $, when restricted to $V_i$, with their matrix representations with respect to the nodal basis $\left\{\il \phi\right\}_{\ell}$. Under this basis, the nodal evaluation operator $w\mapsto w(\il y)$ corresponds to the identity matrix $\Id$.

We now define the numerical scheme for \eqref{eq:HJB_ibvp}. Obtain the numerical solution $v_i(T, \cdot) \in V_i^0$ by nodal interpolation of $v_T$. Then, for each $k \in \{0,\dots,T/h_i-1\}$, the numerical solution $v_i(\ik s, \cdot) \in V_i^0$ is defined inductively by
\begin{align} \label{num}
- d_i v_i(\ik s,\cdot) + \sup_{\alpha\in A} \bigl( \sia E v_i(\iko s,\cdot) + \sia I v_i(\ik s,\cdot) - \sia F \bigr) = 0.
\end{align}

\section{Review of monotonicity and uniform convergence} \label{sec:review}

The proof of gradient convergence in weighted spaces, given in Section \ref{sec:gradient}, is based on the non-negativity of numerical solutions and uniform convergence to the viscosity solution. 

\begin{assumption} \label{thm:monotonicity}
For each $\alpha\in A$, assume that $\sia E$, restricted to $V_i$, has non-positive off-diagonal entries. Let $h_i$ be small enough so that $h_i \sia E - \Id$ is monotone for every $\a$, i.e.~so that all entries of all $h_i \sia E - \Id$ are non-positive. For each $\a$, suppose that for all $v \in V_i$ such that $v$ has a non-positive local minimum at the internal node $\il y$, we have $( \sia I v )_{\ell} \leq 0$.
\end{assumption}

It was shown in \cite[Theorem 3.1]{JS13} that Assumption \ref{thm:monotonicity} implies the existence of a unique numerical solution $v_i$ of \eqref{num} and that $v_i$ is non-negative.

Let $t = \theta s_i^k + (1 - \theta) s_i^{k+1} \in [s_i^k, s_i^{k+1}]$ lie between two time steps, $\theta \in [0,1]$. Then we interpret $v_i(t, \cdot)$ as the linear interpolant between $v_i(s_i^k, \cdot)$ and $v_i(s_i^{k+1}, \cdot)$:
\begin{align} \label{eq:affine}
v_i(t, \cdot) = \theta v_i(s_i^k, \cdot) + (1 - \theta) v_i(s_i^{k+1}, \cdot).
\end{align}
\begin{assumption} \label{thm:uniform}
The Hamilton--Jacobi--Bellman problem \eqref{eq:HJB_ibvp} has a unique viscosity solution $v$ and
\begin{align} \label{conv}
\lim_{i \to \infty} \| v_i - v \|_{L^\infty(\O_T)} = 0.
\end{align}
\end{assumption}

\noindent In \cite{JS13} it was demonstrated that Assumption \ref{thm:uniform}, that is uniform convergence, holds if the following conditions are satisfied:

\begin{enumerate}
\item {\em Orthogonal projection:} \label{orthproj} Suppose there exist linear mappings $\P : \; C([0,T], H^1(\O)) \to [0,T] \times V_i$ which satisfy for all $\hil \phi \in V_i^0$
\begin{equation}\label{eq:ellprojdef}
\langle \nabla \P w(t,\cdot), \nabla \hil \phi \rangle = \langle \nabla w(t,\cdot), \nabla \hil \phi \rangle, \quad \forall t \in [0,T],
\end{equation}
and there is a constant $C\geq 0$ such that for every $ w \in C^{\infty}(\R^{d})$ and $i\in\N$,
\begin{equation}\label{ass:ellprojstab}
\norm{\P w}_{W^{1,\infty}(\O)} \leq C \norm{w}_{W^{1,\infty}(\O)}
\quad\text{and}\quad
\lim_{i\tends \infty} \norm{\P w - w}_{W^{1,\infty}(\O)}=0.
\end{equation}
\item {\em Boundary control:} For each $\a \in A$, we define $\ia v \colon S_i \to V_i^0$ to be the numerical solution of the linear evolution problem associated to the control $\a$ with homogeneous Dirichlet conditions: $\ia v(T,\cdot) = v_i(T,\cdot)$ is the interpolant of $v_T$, and for each $k \in \{0,\dots,T/h_i-1\}$,
\begin{equation} \label{numlinearsol} 
(h_i \sia I + \Id)\, \ia v(\ik s,\cdot) = - (h_i \sia E - \Id)\, \ia v(\iko s,\cdot) + h_i \sia F.
\end{equation}
Suppose that for each $(t,x)\in[0,T]\times\pO$
\[
\inf_{\a\in A} \sup_{(\ik s,\il y) \to (t,x)} \limsup_{i \to \infty} v_i^{\a}(\ik s,\il y) = 0,
\]
where the supremum is taken over the set of all sequences of nodes which converge to $(t,x)$. 
\item {\em Comparison:} Let $\overline{v}$ be a lower semi-continuous supersolution with $\overline{v}(T,\cdot) = v_T$ and $\overline{v}|_{[0,T] \times \pO} = 0$. Similarly, let $\underline{v}$ be an upper semi-continuous subsolution with $\underline{v}|_{[0,T] \times \pO} = 0$ and $\underline{v}(T,\cdot) = v_T$. Then $\underline{v} \le \overline{v}$.
\end{enumerate}

\section{Projection into the approximation space} \label{sec:projection}

For shorthand, let $W = W^{1,d+1+\eps}(\O_T) \cap L^2((0,T), H^1_0(\O)) \subset C(\overline{\O_T})$ with $\eps > 0$. We also use the discrete spaces 
\[
W_i := \{ v \in C([0,T], V_i^0) : v|_{[\ik s, \iko s] \times \O} \text{ is affine in time} \},
\]
which means that functions in $W_i$ have the form of \eqref{eq:affine} between two time-steps. Observe that $W_i \subset W$ for all $i \in \N$.

We introduce the cut-off operation 
\[
C_i : \; W \to W, \; w \mapsto \max \{ w - \| v - v_i \|_{L^\infty(\O_T)} , 0 \}
\]
and denote the nodal interpolant on $[0,T] \times \oO$ by $I_i$. Finally we define $Q_i = I_i \circ C_i$. Thus $Q_i$ is a mapping of the type $W \to W_i$. Observe that $Q_i v \in W_i$ satisfies homogeneous boundary conditions. Furthermore, from $C_i v \le v_i$ and the monotonicity of the nodal interpolation operator it follows that $Q_i v \le v_i$. The stability of the max operation and $I_i$ gives, cf. \cite[Corollary 1.110]{ErnGuermond} with $\ell = 0$ and $p = d+1+\eps$,
\begin{align} \label{eq:Qstab} \begin{array}{ll}
\| Q_i v \|_W & \lesssim \| v \|_W,\\[1mm]
\| Q_i v(T,\cdot) \|_{L^\infty(\O)} & \lesssim \| v(T,\cdot) \|_{L^\infty(\O)}.
\end{array} \end{align}

\begin{lemma} \label{projectionprop}
Suppose that $v \in W$. The sequence $Q_i v$ consists of non-negative functions, satisfying homogeneous Dirichlet boundary conditions and $Q_i v \le v_i$. Moreover, the sequence converges strongly in $W$ and $L^\infty(\O_T)$ to $v$. 
\end{lemma}

\begin{proof} The convergence of $Q_i v$ to $v$ in $W$ remains. We break the proof into two steps by means of the triangle inequality: 
\[
\| Q_i v - v \|_W \le \| Q_i v - I_i v \|_W + \| I_i v - v \|_W.
\]
{\em Step 1.} Observe that
\[
v - C_i v = \min \{ \| v - v_i \|_{L^\infty(\O_T)} , v \}.
\]
Thus $\| v - C_i v \|_{L^\infty(\O_T)} \to 0$ as $i \to \infty$. Moreover,
\[
\| \nabla (v - C_i v) \|_{L^{d+1+\eps}(\O_T)} = \| \nabla v \|_{L^{d+1+\eps}(\Gamma_i)},
\]
where $\Gamma_i = \{ x \in \O_T : v(x) \in (0,\| v - v_i \|_{L^\infty(\O_T)}) \}$ as we have $\nabla (v - C_i v) = 0$ in $\O_T \setminus \Gamma_i$. Because $\bigcap_i \Gamma_i = \emptyset$ it follows that $\| v - C_i v \|_W \to 0$ as $i \to \infty$. Hence, owing to \eqref{eq:Qstab}, $\| Q_i v - I_i v \|_W = \| I_i (C_i v - v) \|_W$ vanishes as $i \to \infty$.

{\em Step 2.} Let $\delta > 0$. Recalling the density of smooth functions in $W$ there is a $v_\delta \in C^\infty(\overline{\O_T})$ such that $\| v_\delta - v \|_W \le \delta$. Owing to the uniform continuity of the partial derivatives of $v_\delta$ on the compact set $\overline{\Omega_T}$ there is a $j \in \N$ such that 
\[
\bigl( \max_{x \in \tau} \partial_k v_\delta(t, x) \bigr) - \bigl( \min_{x \in \tau} \partial_k v_\delta(t, x) \bigr) \le \delta
\]
for all $i \ge j$, $\tau \in \m T_i$, $t \in [0,T]$, and $k \in \{1, \ldots, d\}$. Let $\mu = (\mu_1, \ldots, \mu_d) \in \R^d$ denote vectors with $\| \mu \|_2 = 1$. Then there is a maximiser $x' \in \tau$ and a minimiser $x'' \in \tau$ such that
\begin{align*}
\bigl( \max_{x \in \tau} \partial_\mu v_\delta(t,x) \bigr) - \bigl( \min_{x \in \tau} \partial_\mu v_\delta(t,x) \bigr) & = \partial_\mu v_\delta(t,x') - \partial_\mu v_\delta(t,x'')\\
& =  \sum_k \mu_k \bigl( \partial_k v_\delta(t,x') - \partial_k v_\delta(t,x'') \bigr) \le \sum_k |\mu_k| \delta.
\end{align*}
From the mean value theorem it follows that directional derivatives $\partial_\mu I_i v_\delta$ in the direction of an edge of $\tau$ are attained on that edge by $\partial_\mu v_\delta$ and thus $\| \partial_\mu I_i v_\delta - \partial_\mu v_\delta \|_{L^\infty(\tau)} \le \| \mu \|_1 \delta \le \sqrt{d} \delta$. For each $\tau$ a basis $\mu^1, \ldots, \mu^d$ of $\R^d$ can be formed of unit vectors parallel to edges. Let $A_\tau$ be the matrix expressing the coordinate transformation from $\mu^1, \ldots, \mu^d$ to the canonical basis. Because of shape-regularity there is a constant $C$ which bounds the row-sum norm: $\| A_\tau \|_{\rm row} \le C$ for all $\tau \in \m T_i$ and $i \in \N$. Thus we have $\| \partial_k I_i v_\delta - \partial_k v_\delta \|_{L^\infty(\O_T)} \le C \sqrt{d} \delta$ for $k \in \{1, \ldots, d\}$. We conclude that for $i$ sufficiently large
\[
\| I_i v - v \|_W \le \| I_i (v - v_\delta) \|_W + \| I_i v_\delta - v_\delta \|_W + \| v_\delta - v \|_W \lesssim \delta,
\]
using once again \eqref{eq:Qstab}. \end{proof}

We shall also require a super-approximation result on nodal interpolation in weighted Sobolev spaces. We cite Theorem 2.1 in \cite{Demlow11}, with the notation adapted to this paper.

\begin{lemma} \label{superapprox}
Let $T$ be a mesh element with diameter $\Delta x \le 1$ and $I$ be the nodal interpolation operator of $T$. Given $g \in W^{2,\infty}(T)$ there exists a value $K$, depending on $\| g \|_{W^{2,\infty}(T)}$ and the shape regularity of $T$ and the dimension $d$, such that for {\em affine} functions $w$
\[
\| g^2 w - I (g^2 w) \|_{H^1(T)} \le K \Delta x \bigl( \| \nabla (g w) \|_{L^2(T)} + \| w \|_{L^2(T)} \bigr).
\]
\end{lemma}

\section{Coercivity properties of the Hamilton--Jacobi--Bellman operator and its discretisation} \label{sec:coer}

Owing to the non-negativity of $v$, for each $\ha \in A$, we formally have for the exact solution
\begin{align} \label{varbound}
\partial_t v + \sup_\alpha (L^\alpha v - f^\alpha) = 0 \implies \, \partial_t v + L^\ha v \le f^\ha \implies \langle \partial_t v, v \rangle + \langle L^\ha v, v \rangle \le \langle f^\ha, v \rangle.
\end{align}
Furthermore, if there exists an $\ha \in A$ such that $a^{\ha}\in W^{2,\infty}(\O)$ and $c^{\ha}-\half (\nabla \cdot b^{\ha}+\Delta a^{\ha}) \geq 0$ we have for $w \in H_0^1(\O)$
\begin{align} \label{eq:ibp_pos}
\langle L^\ha w, w \rangle = \tst \langle a^\ha \nabla w, \nabla w \rangle + \langle (c^{\ha}-\half (\nabla \cdot b^{\ha}+\Delta a^{\ha})) w, w \rangle,
\end{align}
thus giving in combination with \eqref{varbound} control on $v$ in a Sobolev space weighted by $a^\ha$. We intend to build the gradient convergence proof upon a bound similar to \eqref{varbound}, with the differential operator replaced by its discretisation and $v$ by $v_i - Q_i v$.

Fix an arbitrary $\alpha \in A$. It is useful to view $\sia E$ and $\sia I$ as bilinear forms on $H^1(\O) \times V_i$. Functions $u \in V_i$ have the nodal representation
\[
u(y) = \sum_{\ell} u(\il y) \, \il \phi(y).
\]
To test with functions other than $\hil \phi$ we introduce the following bilinear form as a partially discrete operation: for $w \in H^1(\O)$ and $u \in V_i$
\[
\llangle \sia E w, u \rrangle := \sum_\ell u(\il y) \bigl( \bia  a(\il y) \langle \nabla w, \nabla \il \phi \rangle + \langle \bia  b \cdot \nabla w + \bia  c \, w, \il \phi \rangle \bigr).
\]
We use the corresponding interpretation for $\llangle \sia I w, u \rrangle$ and also
\begin{align*}
\llangle w, u \rrangle & = \llangle \Id \, w, u \rrangle = \sum_\ell w(\il y) \, u(\il y)\| \il \phi \|_{L^1(\O)},\\
\llangle \sia F, u \rrangle & = \sum_\ell u(\il y) \, \langle f_i^\alpha, \il \phi \rangle = \langle f_i^\alpha, u \rangle.
\end{align*}
Let $H^1_\gamma(\O)$ be the closure of $C_0^\infty(\O)$ in the norm
\[
\| v \|_\gamma^2 := \int_\O v^2 \gamma \, \d x + \int_\O |\nabla v|^2 \gamma \, \d x,
\]
where $\gamma: \O \to \R$ is a non-negative $L^\infty(\O)$ function. We write $H^1_\a(\O)$ as abbreviation of $H^1_{a^\a}(\O)$ and $H^1_i(\O)$ for $H^1_{\gamma_i}(\O)$ where $\gamma_i$ is a weight depending on $i \in \N$. 

We now formulate a discrete analogue of \eqref{varbound} with \eqref{eq:ibp_pos}: Consider that there exists an $\a \in A$ and weights $\gamma_i$ and a $C' > 0$ such that for all $i \in \N$
\begin{align}  \nonumber
& | w |_{L^2((0,T), H_i^1(\O))}^2 \\ \nonumber
\lesssim \, & \sum_{k = 0}^{\frac{T}{h_i} - 1} \Bigl( \bllangle \bigl(h_i \sia E - \Id \bigr) w(\iko s,\cdot) + \bigl(h_i \sia I + \Id \bigr) w(\ik s,\cdot), w(\ik s,\cdot) \brrangle \Bigr)\\
& \qquad + \tst \oh \llangle w(T,\cdot), w(T,\cdot) \rrangle + C' \| \nabla w(T,\cdot) \|_W^2 \label{eq:posdef} \\ \nonumber
\stackrel{(*)}{=} \, & \sum_{k = 0}^{\frac{T}{h_i}-1} \Bigl( h_i \bllangle \sia E w(\iko s,\cdot) + \sia I w(\ik s,\cdot), w(\ik s,\cdot) \brrangle \tst
 + \tst \oh \llangle w(\iko s,\cdot) - w(\ik s,\cdot), w(\iko s,\cdot) - w(\ik s,\cdot) \rrangle \Bigr)\\ & \qquad + \tst \oh \llangle w(0,\cdot), w(0,\cdot) \rrangle + C' \| \nabla w(T,\cdot) \|_W^2 \nonumber
\end{align}
for all $w \in W_i$ with $w \geq 0$ and $i \in \N$, where $(*)$ is a simple reformulation in terms of a telescope sum.

\begin{example} \label{ex:implicit}
Suppose that there is an $\alpha \in A$ such that  $\sqrt{a^{\a}}\in W^{2,\infty}(\O)$ and $c^{\a}-\half (\nabla \cdot b^{\a}+\Delta a^{\a}) \geq 0$. Choosing a fully implicit scheme with $\ia I = L^\a + 2 K \, \Delta x_i \, (\| \sqrt{a^\a} \|_{W^{1,\infty}(\O)} + 1) \, \Delta$ and $E^\a = 0$, the highest order term in $\llangle \sia I w, w \rrangle$ is at time $\ik s$:
\begin{align} \label{hot}
\sum_{\ell} w(\ik s, \il y) \bbia a(\ik s, \il y) \pair{\nabla w(\ik s, \cdot)}{\nabla \il \phi}=\pair{\nabla w(\ik s, \cdot)}{\nabla I_i (\bbia a(\ik s, \cdot) w(\ik s, \cdot))}
\end{align}
with $w \in W_i$ as well as weight and numerical diffusion coefficient 
\[
\gamma_i := \bbia a = a^\a + 2 K \, \Delta x_i \, (\| \sqrt{a^\a} \|_{W^{1,\infty}(\O)} + 1).
\]
According to Lemma \ref{superapprox}, for $i$ sufficiently large,
\begin{align} \nonumber
|\pair{\nabla w}{\nabla I_i (a^{\a} w)} - \pair{\nabla w}{\nabla a^{\a} w}| \le \; & \| \nabla w \|_{L^2(\O)} \cdot \| I_i (a^{\a} w) - a^{\a} w \|_{H^1(\O)}\\ 
\le \; & \textstyle K \; \Delta x_i \, \| \nabla w \|_{L^2(\O)} \bigl( \| \nabla (\sqrt{a^\a} w) \|_{L^2(\O)} + \| w \|_{L^2(\O)}\bigr) \label{coersuper}\\
\le \; & \textstyle K \; \Delta x_i \, (\| \sqrt{a^\a} \|_{W^{1,\infty}(\O)} + 1) \, \| w \|_{H^1(\O)}^2.\nonumber
\end{align}
Therefore, adding and subtracting $\pair{\nabla w}{\nabla a^{\a} w}$ from \eqref{hot}, the triangle inequality and \eqref{coersuper} give
\begin{align*}
\bllangle \sia I w, w \brrangle = \, & \bigl( \langle \nabla w, \nabla I_i (a^\a w) - \pair{\nabla w}{\nabla a^{\a} w} \bigr) + \pair{\nabla w}{\nabla a^{\a} w}\\
& + 2 K \, \Delta x_i \, (\| \sqrt{a^\a} \|_{W^{1,\infty}(\O)} + 1) \, \pair{\nabla w}{\nabla w} + \langle b^\a \cdot \nabla w + c^\a \, w, w \rangle\\
\ge \, & \langle L^\a w, w \rangle + K \, \Delta x_i \, (\| \sqrt{a^\a} \|_{W^{1,\infty}(\O)} + 1) \, \pair{\nabla w}{\nabla w} \ge \oh \abs{w}^2_{H^1_i(\O)},
\end{align*}
implying \eqref{eq:posdef} as the reformulation $(*)$ shows. We used here that $\bbia a - a^\a$ is constant and therefore unaffected by nodal interpolation. \hfill $\square$
\end{example}

Due to the definition of the numerical method and the non-negativity of the $v_i$, if \eqref{eq:posdef} holds then
\begin{align} \nonumber
| v_i |_{L^2(H^1_i)}^2 \lesssim \, & \sum_{k = 0}^{\frac{T}{h_i}-1} \Bigl( \bllangle \bigl(h_i \sia E - \Id \bigr) v_i(\iko s,\cdot) + \bigl(h_i \sia I + \Id \bigr) v_i(\ik s,\cdot), v_i(\ik s,\cdot) \brrangle \Bigr)\\
& \qquad + \tst \oh \llangle v_i(T,\cdot), v_i(T,\cdot) \rrangle + C' \| v_i(T,\cdot) \|_W^2  \nonumber\\
\le \, & \sum_{k = 0}^{\frac{T}{h_i}-1} \bllangle h_i \sia F, v_i(\ik s,\cdot) \brrangle + \tst \oh \llangle v_i(T,\cdot), v_i(T,\cdot) \rrangle + C' \| v_i(T,\cdot) \|_W^2  \nonumber \\
\lesssim & \, (T \, \| f_i^\alpha \|_{L^1(\O)} +1 ) \, \| v_i \|_{L^\infty([0,T] \times \O)} + C' \| v(T,\cdot) \|_W^2, \label{eq:sob_bound}
\end{align}
using stability of nodal interpolation to bound $\| v_i(T,\cdot) \|_W$. So \eqref{eq:posdef} guarantees stability in the $L^2(H^1_\gamma)$-norm.

\section{Weighted gradient convergence} \label{sec:gradient}

Let us assume for a moment that the approximations $Q_i v \in W_i$ satisfy
\begin{align}
\label{eq:projcond}
\sum_{k = 0}^{\frac{T}{h_i}-1} \bllangle \bigl(h_i \sia E - \Id \bigr) Q_i v(\iko s,\cdot) + \bigl(h_i \sia I + \Id \bigr) Q_i v(\ik s,\cdot), (v_i - Q_i v)(\ik s, \cdot) \brrangle \to 0.
\end{align}
We discuss the validity of \eqref{eq:projcond} in the proof of Theorem \ref{thm:main}. 

With $\xi^k = v_i(\ik s,\cdot) - Q_i v(\ik s,\cdot)$,
\begin{align} \nonumber
& | v_i - Q_i v |_{L^2(H_i^1)}^2\\ \nonumber
\stackrel{\eqref{eq:posdef}}{\lesssim} \, & \sum_{k = 0}^{\frac{T}{h_i}-1} \bllangle \bigl(h_i \sia E - \Id \bigr) \xi^{k+1} + \bigl(h_i \sia I + \Id \bigr) \xi^k, \xi^k \brrangle + \tst \oh \llangle \xi^{T/h_i}, \xi^{T/h_i} \rrangle + C' \| \xi^{T/h_i} \|_W^2\\ \nonumber
= \, & \sum_{k = 0}^{\frac{T}{h_i}-1} \bllangle \bigl(h_i \sia E - \Id \bigr) v_i(\iko s,\cdot) + \bigl(h_i \sia I + \Id \bigr) v_i(\ik s,\cdot), \xi^k \brrangle + \tst \oh \llangle \xi^{T/h_i}, \xi^{T/h_i} \rrangle\\ \nonumber
& - \sum_{k = 0}^{\frac{T}{h_i}-1} \bllangle \bigl(h_i \sia E - \Id \bigr) Q_i v(\iko s,\cdot) + \bigl(h_i \sia I + \Id \bigr) Q_i v(\ik s,\cdot), \xi^k \brrangle +  C' \| \xi^{T/h_i} \|_W^2\\ \nonumber
\stackrel{(*)}{\le} \, & \sum_{k = 0}^{\frac{T}{h_i}-1} \bllangle h_i \sia F, \xi^k \brrangle \; -
\sum_{k = 0}^{\frac{T}{h_i}-1} \bllangle \bigl(h_i \sia E - \Id \bigr) Q_i v(\iko s,\cdot) + \bigl(h_i \sia I + \Id \bigr) Q_i v(\ik s,\cdot), \xi^k \brrangle\\[2mm]
& + \tst \oh \llangle \xi^{T/h_i}, \xi^{T/h_i} \rrangle +  C' \| \xi^{T/h_i} \|_W^2, \label{eq:Hone_split}
\end{align}
using in $(*)$ the numerical scheme and that, due to the assumptions on the $Q_i$, the sign of $v_i - Q_i v$ is known. Since
\begin{align*}
\sum_{k = 0}^{\frac{T}{h_i}-1} \bllangle h_i \sia F, \xi^k \brrangle \le \, & \| f_i^\alpha \|_{L^2} \! \sum_{k = 0}^{\frac{T}{h_i}-1} h_i \bigl (\| v_i(\ik s, \cdot) - v(\ik s, \cdot) \|_{L^2} + \| v(\ik s, \cdot) - Q_i v(\ik s, \cdot) \|_{L^2} \bigr)\\
\lesssim \, & \| f_i^\alpha \|_{L^2(\O)} \, \bigl( \| v_i - v \|_{L^2(\O_T)} + \| v - Q_i v \|_{L^2(\O_T)} \bigr),
\end{align*}
the first term in \eqref{eq:Hone_split} vanishes as $i \to \infty$. The second term vanishes due to \eqref{eq:projcond}. For the two last terms recall Step 2 in the proof of Lemma \ref{projectionprop} and that $v_i$ is the interpolant of $v$ at time $T$. Hence
\[
| v_i - v |_{L^2([0,T], H_i^1(\O))} \to 0
\]
as $i \to \infty$.

\medskip

\begin{theorem} \label{thm:main}
Suppose the value function $v$ belongs to the space $W$ and there is an $\alpha \in A$ and weights $\gamma, \gamma_i \in L^\infty(\O)$, $i \in \N$, such that
\begin{enumerate}
\item $\| \sqrt{\bia a} \|_{W^{2,\infty}(\O)}$ and $\| \sqrt{\bbia a} \|_{W^{2,\infty}(\O)}$ are uniformly bounded in $i$,
\item there is a $C > 0$ such that $\bia a \le C \gamma_i$ and $\bbia a \le C \gamma_i$ for all $i \in \N$,
\item the coercivity condition \eqref{eq:posdef} is satisfied for large $i$ and $\gamma \lesssim \gamma_i$.
\end{enumerate}
Then the numerical solutions converge to the viscosity solution $v$ strongly in $L^2([0,T], H_\gamma^1(\O))$.
\end{theorem}

\begin{proof}
It remains to show \eqref{eq:projcond}. The terms connected to the time derivative in \eqref{eq:projcond} vanish in the limit as
\begin{align} \nonumber
\sum_{k = 0}^{\frac{T}{h_i}-1} \bllangle Q_i v(\iko s,\cdot) - Q_i v(\ik s,\cdot), \xi^k \brrangle = & \sum_{k = 0}^{\frac{T}{h_i}-1} h_i \bllangle (\partial_t Q_i v)|_{(\ik s,\iko s)}, \xi^k \brrangle\\
\lesssim & \; \| \partial_t v \|_{L^1(\Omega_T)} \, \| \xi^k \|_{L^\infty(\Omega_T)},
\end{align}
using the uniform convergence in $\xi^k$. Recall that $\llangle \sia I Q_i v(\ik s, \cdot), \xi^k \rrangle$ is equal to 
\begin{align*}
\sum_\ell (v_i - Q_i v)(\ik s, \il y) \Bigl( \bbia  a(\il y) \langle \nabla Q_i v(\ik s, \cdot), \nabla \il \phi \rangle + \langle \bbia  b \cdot \nabla Q_i v(\ik s, \cdot) + \bbia  c \, Q_i v(\ik s, \cdot), \il \phi \rangle \Bigr).
\end{align*}
The lower-order terms vanish due to the uniform convergence of $ v_i -Q_i v$ to $0$ and the bound
\[
\sup_i \| \bbia  b \cdot \nabla Q_i v(\ik s, \cdot) + \bbia  c \, Q_i v(\ik s, \cdot) \|_{L^1(\O)} < \infty.
\]
We note that
\begin{align*}
\sum_\ell (v_i - Q_i v)(\ik s, \il y) \, \bbia  a(\il y) \, \langle \nabla Q_i v(\ik s, \!\cdot), \! \nabla \il \phi \rangle = \langle \nabla Q_i v(\ik s, \!\cdot\!), \!\nabla I_i (\bbia a(v_i - Q_i v))(\ik s, \!\cdot) \rangle,
\end{align*}
so that in \eqref{eq:projcond} the implicit part of the second-order term becomes
\begin{align} \label{eq:implicit}
\sum_{k = 0}^{\frac{T}{h_i}-1} \!\!\! h_i \, \langle \nabla Q_i v(\ik s, \cdot), \nabla I_i (\bbia a(v_i - Q_i v))(\ik s, \cdot) \rangle
= \int_0^T \! \langle J_i \nabla Q_i v, J_i\nabla I_i (\bbia a(v_i - Q_i v))\rangle \, \d t,
\end{align}
where $J_i$ maps any $w : [0,T] \to L^2(\O;\R^d)$ onto the step function $(J_i w)|_{[\ik s, \iko s)} \equiv w(\ik s, \cdot)$. Observe that $J_i \nabla Q_i v$ converges strongly in $L^2(\O_T; \R^d)$. Owing to Lemma \ref{superapprox}, we have the chain of inequalities
\begin{align*}
\| \nabla I_i (\bbia a \xi) \|_{L^2(\O; \R^d)} & \le \, \| \nabla (\bbia a \xi) \|_{L^2(\O; \R^d)} + \| \nabla (\bbia a \xi - I_i (\bbia a \xi)) \|_{L^2(\O; \R^d)}\\
& \textstyle \hspace{-1cm} \le \, \| \nabla (\bbia a \xi) \|_{L^2(\O; \R^d)} + K \; \Delta x_i \bigl( \| \nabla ( \sqrt{\bbia a} \xi) \|_{L^2(\O)} + \| \xi \|_{L^2(\O)} \bigr)\\
& \textstyle \hspace{-1cm} \le \, (\| \sqrt{\bbia a}\|_{W^{1,\infty}(\O)} + K \; \Delta x_i) \| \sqrt{\bbia a} \nabla \xi \|_{L^2(\O; \R^d)} + (\| \sqrt{\bbia a}\|_{W^{1,\infty}(\O)}^2 + K \; \Delta x_i) \| \xi \|_{L^2(\O)}
\end{align*}
at a time $\ik s \in [0,T)$. In combination with Assumption \ref{thm:uniform} as well as \eqref{eq:Qstab} and \eqref{eq:sob_bound} this gives an $L^\infty(L^2)$ bound over $J_i\nabla I_i (\bbia a(v_i - Q_i v))$ in \eqref{eq:implicit}. The convergence 
\[
\lim_{i \to \infty} \; \int_0^T \langle w, J_i \nabla I_i (\bbia a(v_i - Q_i v))\rangle \d t = - \lim_{i \to \infty} \; \int_0^T \langle \nabla \cdot w, J_I I_i (\bbia a(v_i - Q_i v))\rangle \d t = 0
\]
with test functions $w$ in the dense subset $C_0^1(\O_T; \R^d)$ gives weak convergence of $\nabla I_i (\bbia a(v_i - Q_i v))$ in $L^2(\O_T; \R^d)$, see \cite[p.~121]{Yoshida}. Combining weak and strong convergence \cite[Prop.~21.23]{ZeidlerII}, it is ensured that \eqref{eq:implicit} converges to $0$ as $i \to \infty$. A similar argument shows that $\sum_k h_i \, \llangle \sia E Q_i v(\iko s, \cdot), \xi^k \rrangle$ vanishes in the limit. Therefore we proved \eqref{eq:projcond}.
\hfill \end{proof}

The Sobolev regularity of the value functions is for example discussed in \cite{FS06} and \cite{WWZ} and \cite{Zhou}.

\section{The method of artificial diffusion} \label{sec:art_diff}

We illustrate now a way of choosing the coefficients of $\sia E$ and $\sia I$ in order to satisfy the assumptions of the above analysis.

For all $\alpha \in A$ we need to impose the conditions to ensure uniform convergence. For one $\ha \in A$ we wish to enforce terms to guarantee convergence in a Sobolev space with a weight associated to $L^\ha$.

\begin{remark}
One could define a set $B \subset A$ of multiple indices $\ha$ for which one wishes to derive Sobolev norm bounds, which are associated to different weights $\gamma^\ha$ and $\gamma_i^\ha$. The analysis generalises directly.
\end{remark}

Given a function $g : \Omega \to \R^d$ and an element $K$ of the mesh $\calT_i$, we denote
\[
|g|_K:= \Bigl( \sum_{j=1}^{d} \bigl\| g_j \bigr\|_{L^\infty(K)}^2 \Bigr)^{\half}.
\]
If $g$ is elementwise constant then $|g|_K$ is simply the Euclidean norm of $g$ on $K$. Let $\Delta x_K$ denote the diameter of $K$. We assume that the meshes $\calT_i$ are strictly acute, cf.~\cite{BE02}, in the sense that there exists $\theta \in (0,\pi/2)$ such that for all $i \in \N$:
\begin{equation}\label{eq:acutemesh}
\nabla \il \phi \cdot \nabla \phi_{i}^{l} \bigl|_K \leq - \, \sin(\theta) \; | \nabla \il \phi |_K \; | \nabla \phi^{l}_i |_K \qquad \forall \ell,l \leq \dim V_i,\;\ell \neq l, \; \forall K\in\calT_i.
\end{equation}
We choose a splitting of the form $a^{\a} = \tia a + \ttia a$, $b^{\a} = \tia b + \ttia b$, $c^{\a} = \tia c + \ttia c$, which does not depend on $i \in \N$. It is generally necessary to add artificial diffusion to the second-order coefficients, that means that we generally need to determine $i$-dependent coefficients $\bia a \ge \tia a$ and $\bbia a \ge \ttia a$. It can also be necessary to construct $\bbia c \ge \ttia c$. For the other lower-order terms we can use the above splitting directly: For all $i \in \N$ set
\[
\bia b = \tia b, \bbia b = \ttia b, \bia c = \tia c \text{ and also } f_i^\alpha=f^\alpha,
\]
where all terms are in $C(\oO)$, $\tia a$ and $\ttia a$ are non-negative and all $ \tia c$ and $ \ttia c$ are non-negative and satisfy inequality \eqref{react}.

For instance, one could discretise symmetric terms implicitly, i.e.~$a^{\a} = \ttia a$ and $c^{\a} = \ttia c$, and screw-symmetric terms explicitly, i.e.~$b^{\a} = \tia b$. With this approach the coercivity properties of $L^\alpha$ are well incorporated. Alternatively, if there is coercivity which is uniform in $\alpha$ with respect to a useful weight $\gamma$, then it is interesting to select $\sia I = {\sf I}_i$ independently of $\alpha$ because in this case only a linear system needs to be solved at each time step while an ${\sf O}(\Delta x_i)$ time-step may be preserved. We also refer to \cite{JS12} for further illustrations of operator splittings.

To obtain uniform convergence we select the non-negative parameters $ \bia \nu $ and $ \bbia \nu $ such that for all $K$ that have $\il y$ as vertex:
\begin{subequations}\label{ineq:nut}
\begin{align}
\bigl( |  \bia b |_K \, +  \Delta x_K \|  \bia c \|_{L^\infty(K)} \bigr) \le \, &  \bia \nu \, \sin(\theta) \, | \nabla \hil \phi |_K \, {\rm vol}(K), \\
\bigl( | \bbia b |_K \, +  \Delta x_K \| \bbia c \|_{L^\infty(K)} \bigr) \le \, & \bbia \nu \, \sin(\theta) \, | \nabla \hil \phi |_K \, {\rm vol}(K).
\end{align}
\end{subequations}

For all $\a \in A \setminus \{ \ha \}$ we choose $ \bia a $ and $\bbia a $ both in $C(\oO)$ such that at the nodes
\begin{align} \label{diffusion}
\left. \begin{array}{rl}
\bia a(\il y) \hspace*{-3mm} & \geq \tia a(\il y) + \bia \nu\\[1mm]
\bbia a(\il y) \hspace*{-3mm} & \geq \ttia a(\il y) + \bbia \nu
\end{array} \right\}
\end{align}
Moreover, we set $\bbia c = \ttia c$.

We now turn to convergence in the weighted Sobolev norm associated with the control $\ha$. We require that $\sqrt{\tha a}, \sqrt{\ttha a} \in W^{2,\infty}(\O)$ as well as
\begin{align} \nonumber
\tst \ttha c - \oh (\nabla \cdot \ttha b + \Delta \ttha a) & \ge 0,\\
\tst \tha c - \oh (\nabla \cdot \tha b + \Delta \tha a) & \ge 0, \label{Sobolevcond}\\
\gamma := \ttha a - \tha a & \ge 0 \nonumber
\end{align}
with $\tha a \lesssim \gamma$ and $\ttha a \lesssim \gamma$. We choose $ \biha a $ and $\bbiha a $ both in $C(\oO)$ such that at the nodes
\begin{align} \label{Sob_diffusion}
\left. \begin{array}{rl}
\biha a(\il y) \hspace*{-3mm} & \geq \tha a(\il y) + \biha \nu\\[1mm]
\bbiha a(\il y) \hspace*{-3mm} & \geq \ttha a(\il y) + \mu_i
\end{array} \right\}
\end{align}
such that for large $i$
\begin{align} \label{def:mu}
\mu_i \ge \max \Bigl\{ \bbiha \nu, 2 \, K \, \Delta x_i \, \Bigl( \| \sqrt{\biha a} \|_{W^{1,\infty}(\O)} + \| \sqrt{\bbiha a} \|_{W^{1,\infty}(\O)} + 2 \Bigr) + h_i \| \nabla \tia a + \tha b \|_{L^\infty}^2 \Bigr\}.
\end{align}
Notice the recursive nature of the definition of $\bbiha a$ which appears also on the right-hand side of \eqref{def:mu}. For $\Delta x_i$ small enough it is clear that $\mu_i$ can be chosen so that is satisfies \eqref{def:mu}. Finally, we set 
\[
\bbiha c = \ttha c + K \, \Delta x_i \, \Bigl( \| \sqrt{\biha a} \|_{W^{1,\infty}(\O)} + 1 \Bigr) + h_i \| \tha c \|_{L^\infty}^2.
\]
Sections \ref{ver_uniform} and \ref{ver_Sobolev} below show that, under suitable time-step restrictions, \eqref{diffusion} gives convergence.

\subsection{Verification of uniform convergence} \label{ver_uniform}

Suppose that $ w \in V_i$ has a non-positive local minimum at an interior node $ \il y$. It was shown in \cite{JS13} that then
\begin{align} \label{av_mono}
(\sia E w)_\ell \le 0, \qquad (\sia I w)_\ell \le 0,
\end{align}
and, if $\bia \nu$ and $\bbia \nu$ are chosen optimally, then for $K \in \calT_i$
\begin{align} \label{eq:artdifforder} 
\left. \begin{array}{rl}
\bia \nu \hspace*{-3mm} & =  \mathsf{O} \big( \sup_{K} \bigl\{ | \bia b|_K \Delta x_K + \| \bia c \|_{L^\infty(K)} \Delta x_K^2 \bigr\} \big),\\[1mm]
\bbia \nu \hspace*{-3mm} & =  \mathsf{O}\big( \sup_{K} \bigl\{ | \bbia b|_K \Delta x_K + \| \bbia c \|_{L^\infty(K)} \Delta x_K^2 \bigr\} \big).
\end{array} \right\}
\end{align}
Note that \eqref{eq:artdifforder} is consistent with Assumption \ref{consistency}. We turn to time step restrictions for semi-implicit and explicit methods which give Assumption \ref{thm:monotonicity}. The non-positivity of the diagonal terms of $h_i \sia E - \Id$ expands to
\begin{align*}
1 \ge \, & h_i \Bigl( \bia  a(\il y) \langle \nabla \il \phi, \nabla \hil \phi \rangle + \langle \bia  b \cdot \nabla \il \phi + \bia  c \, \il \phi, \hil \phi \rangle \Bigr)\\
= \, & h_i \Bigl( {\mathsf O} \bigl( \bia a \, \Delta x_K^{-2} \bigr) + {\mathsf O} \bigl( | \bia b |_K \, \Delta x_K^{-1} \bigr) + {\mathsf O} \bigl( \bia c \bigr) \Bigr).
\end{align*}
Therefore the time step restriction imposed by $L^\alpha$ is $h_i \lesssim \inf_K \bigl( \Delta x_K^2 / \bia a(\il y)\bigr)$, $\il y \in \overline K$, if there is a non-zero $\tia a$ and $i$ is large. It is $h_i \lesssim \inf_K \bigl( \Delta x_K / |\bia b(\il y)|_K \bigr)$ if all $\bia a = 0$, $i \in \N$, and there are non-zero $\bia b$, and is $\mathsf{O}(1)$ if all $\bia a$ and $\bia b$ vanish. There is no restriction if also all $\bia c$ are zero. If the scheme is fully implicit, there are no time-step restrictions.

Assumption \ref{thm:uniform} holds if there is an orthogonal projection, boundary control and comparison, see page \pageref{orthproj}. The former is essentially a quasi-uniformity assumption on the mesh, cf. \cite{DLSW}, the latter two on the boundary value problem. In fact, the comparison principle is one of the building blocks in the theory of viscosity solutions.

\subsection{Verification of Sobolev convergence} \label{ver_Sobolev}

The main step of this section is the proof that \eqref{eq:posdef} is satisfied. With $\ko w = w(\iko s,\cdot)$ and $\kz w = w(\ik s,\cdot)$,
\begin{align} 
\bllangle \siha E \ko w, \kz w \brrangle = 
\langle \nabla \ko w, \nabla I_i (\biha a \kz w) \rangle + \langle \tha b \cdot \nabla \ko w + \tha  c \, \ko w, \ko w + (\kz w - \ko w) \rangle,
\end{align}
using the interpolation operator as in \eqref{hot}. We also find, as in \eqref{coersuper}, for large $i$
\begin{align*}
| \langle \nabla \ko w, \nabla I_i (\biha a \kz w) \rangle - \langle \nabla \ko w, \nabla \biha a \kz w \rangle | \le \textstyle K \; \Delta x_i \, \big( \big\| \sqrt{\biha a} \big\|_{W^{1,\infty}(\O)} + 1\big) \, | \ko w |_{H^1(\O)} \, \| \kz w \|_{H^1(\O)}.
\end{align*}
Therefore,
\begin{align*}
\langle \nabla \ko w, \nabla I_i (\biha a \kz w) \rangle \ge \, & \, \langle \nabla \ko w, \biha a \nabla \kz w \rangle + \langle \nabla \ko w,  \ko w \nabla \biha a \rangle + \langle \nabla \ko w, (\kz w - \ko w) \nabla \biha a \rangle\\[1mm]
& - \textstyle K \; \Delta x_i \, \big( \big\| \sqrt{\biha a} \big\|_{W^{1,\infty}(\O)} + 1\big) \, | \ko w |_{H^1(\O)} \, \| \kz w \|_{H^1(\O)}.
\end{align*}
Recalling \eqref{Sobolevcond} and using that $\tha a - \biha a$ is constant, we see that
\begin{align*}
\langle \nabla \ko w,  \ko w \nabla \biha a \rangle + \langle \tha b \cdot \nabla \ko w + \tha  c \, \ko w, \ko w \rangle =
\langle \tst (\tha c - \oh (\nabla \cdot \tha b + \Delta \tha a)) \ko w, \ko w \rangle \ge 0.
\end{align*}
Now
\begin{align*}
& | \langle (\nabla \biha a + \tha b) \cdot \nabla \ko w + \tha c \, \ko w, \kz w - \ko w \rangle | \\
\le & \tst \frac{h_i}{2} \bigl( \| \nabla \tia a + \tha b \|_{L^\infty}^2 \, \| \nabla \ko w \|_{L^2}^2 +  \| \tha c \|_{L^\infty}^2 \, \| \ko w \|_{L^2}^2 \bigr) + \frac{1}{2 h_i} \| \kz w - \ko w \|_{L^2}^2\\
\le & \tst \frac{h_i}{2} \bigl( \| \nabla \tia a + \tha b \|_{L^\infty}^2 \, \| \nabla \ko w \|_{L^2}^2 +  \| \tha c \|_{L^\infty}^2 \, \| \ko w \|_{L^2}^2 \bigr) + \frac{1}{2 h_i} \bllangle \kz w - \ko w, \kz w - \ko w \brrangle
\end{align*}
where we can use Jensen's inequality as $\sum_\ell \il \phi (x) = 1$ is a convex combination:
\begin{align*}
\| \kz w - \ko w \|_{L^2}^2 = & \int_\O \Bigl( \sum_\ell \bigl( \kz w(\il y) - \ko w(\il y) \bigr) \il \phi(x) \Bigr)^2 \d x\\
\le & \int_\O \sum_\ell \bigl(\kz w(\il y) - \ko w(\il y) \bigr)^2 \il \phi(x) \d x = \bllangle \kz w - \ko w, \kz w - \ko w \brrangle.
\end{align*}
We summarise
\begin{align} \label{exp_terms}
\bllangle \siha E \ko w, \kz w \brrangle \ge \, & \tst - \oh \big\| \sqrt{\biha a} \nabla \ko w \big\|_{L^2}^2 - \oh \big\| \sqrt{\biha a} \nabla \kz w \big\|_{L^2}^2\\ \nonumber
& - \textstyle K \; \Delta x_i \, \big(\big\| \sqrt{\biha a} \big\|_{W^{1,\infty}(\O)} + 1\big) \, \bigl( \tst \oh | \ko w |_{H^1(\O)}^2  + \oh \| \kz w \|_{H^1(\O)}^2 \bigr) \\
& - \tst \frac{h_i}{2} \bigl( \| \nabla \tia a + \tha b \|_{L^\infty}^2 \, \| \nabla \ko w \|_{L^2}^2 +  \| \tha c \|_{L^\infty}^2 \, \| \ko w \|_{L^2}^2 \bigr) - \frac{1}{2 h_i} \bllangle \kz w - \ko w, \kz w - \ko w \brrangle. \nonumber
\end{align}
As in Example \ref{ex:implicit}, for large $i$,
\begin{align} \label{imp_terms}
\bllangle \siha I \kz w, \kz w \brrangle \ge & \; 
\langle  \ttha a \Delta \kz w, \Delta \kz w \rangle + \langle (\ttha c - \tst \oh (\nabla \cdot \ttha b + \Delta \ttha a) ) \kz w, \kz w \rangle\\ \nonumber
& + \Bigl( K \, \Delta x_i \, \Bigl( 2 \, \| \sqrt{\biha a} \|_{W^{1,\infty}} + \| \sqrt{\bbiha a} \|_{W^{1,\infty}} + 3 \Bigr) + h_i \| \nabla \tia a + \tha b \|_{L^\infty}^2 \Bigr) \langle \nabla \kz w, \nabla \kz w \rangle\\
& + \Bigl( \oh K \, \Delta x_i \, \Bigl( \| \sqrt{\biha a} \|_{W^{1,\infty}(\O)} + 1 \Bigr) + \oh h_i \| \tha c \|_{L^\infty}^2 \Bigr) \langle \kz w, \kz w \rangle. \nonumber
\end{align}
We observe that all terms on the right-hand side of \eqref{exp_terms} can be bounded by corresponding terms in \eqref{imp_terms} of time $k$ or $k+1$---except if $w$ is evaluated at the final time. We choose $C'$ such that
\begin{align*}
C' \|  w^{T/h_i} \|_W^2 \ge \; & \tst \oh \big\| \sqrt{\biha a} \nabla w^{T/h_i} \big\|_{L^2}^2 + \oh \textstyle K \; \Delta x_i \, \big(\big\| \sqrt{\biha a} \big\|_{W^{1,\infty}(\O)} + 1\big) | w^{T/h_i} |_{H^1(\O)}^2 + h_i \| \gamma_i \nabla  w^{T/h_i} \|_{L^2}^2\\
& + \tst \frac{h_i}{2} \bigl( \| \nabla \tia a + \tha b \|_{L^\infty}^2 \, \| \nabla w^{T/h_i} \|_{L^2}^2 +  \| \tha c \|_{L^\infty}^2 \, \| w^{T/h_i} \|_{L^2}^2 \bigr).
\end{align*}
Then we obtain \eqref{eq:posdef} with $\gamma_i = \gamma + \tst \frac{\mu_i}{2}$:
\begin{align*}  
& \sum_{k = 0}^{\frac{T}{h_i}-1} \Bigl( h_i \bllangle \sia E \ko w + \sia I \kz w, \kz w \brrangle \tst
 + \tst \oh \llangle \ko w - \kz w, \ko w - \kz w \rrangle \Bigr) + \tst \oh \llangle w^0, w^0 \rrangle + C' \| w^{T/h_i} \|_W^2\\
\ge & \sum_{k = 0}^{\frac{T}{h_i}} \Bigl( h_i \langle \gamma_i \nabla w^k, \nabla w^k \rangle \Bigr) + \tst \oh \llangle w^0, w^0 \rrangle \ge | w |_{L^2((0,T), H_i^1(\O))}^2,
\end{align*}
where the last inequality follows because the composite trapezium rule is exact for functions in $W_i$. Finally, observe that $\gamma \lesssim \gamma_i$ and that there is a $C > 0$ such that $\bia a \le C \gamma_i$ and $\bbia a \le C \gamma_i$ for all $i \in \N$, as required in the statement of Theorem \ref{thm:main}.


\end{document}